\documentclass[a4paper,10pt]{amsart}
\usepackage[utf8]{inputenc}
\usepackage{amsxtra}
\usepackage{amsopn}
\usepackage{amsmath,amsthm,amssymb}
\usepackage{amscd}
\usepackage{amsfonts}
\usepackage{latexsym}
\usepackage{verbatim}
\usepackage{MnSymbol}
\usepackage{xcolor}

\newcommand{\la}{\llangle}
\newcommand{\ra}{\rrangle}
\newcommand{\lv}{\lVert}
\newcommand{\rv}{\rVert}
\theoremstyle{plain}
\newtheorem{theorem}{Theorem}[section]
\newtheorem*{theorem*}{Theorem}

\newtheorem{lemma}[theorem]{Lemma}

\newtheorem{corollary}[theorem]{Corollary}
\newtheorem{remark}[theorem]{Remark}

\newtheorem*{mt*}{Main Theorem}
\newcommand\C{{\mathbb C}}

\renewcommand\phi{{\varphi}}

\newcommand\N{{\mathbb N}}

\newcommand\R{{\mathbb R}}

\renewcommand\H{{\mathcal H}}
\newcommand{\de}[2]{\frac{\partial #1}{\partial #2}}

\newcommand{\del}{\partial}
\newcommand{\delbar}{\overline{\del}}

\newcommand{\cinf}{\mathcal{C}^\infty}
\DeclareMathOperator{\supp}{supp}

\let\sup\undefined
\DeclareMathOperator*{\sup}{sup\vphantom{p}}
\DeclareMathOperator{\vol}{Vol}

\DeclareMathOperator{\Ker}{Ker}

\DeclareMathOperator{\D}{\mathcal{D}}

\DeclareMathOperator{\rank}{rank}
\DeclareMathOperator{\Hom}{Hom}
\DeclareMathOperator{\lV}{\lVert}
\DeclareMathOperator{\rV}{\rVert}
\let\phi\varphi
\let\c\overline

\title{$W^{1,2}$ Bott-Chern and Dolbeault Decompositions on
K\"ahler Manifolds}

\author{Riccardo Piovani}
\address{Dipartimento di Matematica\\
Universit\`{a} di Pisa\\
Largo Bruno Pontecorvo 5 \\
56127 Pisa, Italy}
\email{riccardo.piovani@phd.unipi.it}


\begin{document}
\maketitle

\begin{abstract}
Let $(M,J,g,\omega)$ be a K\"ahler manifold. We prove a $W^{1,2}$ weak Bott-Chern decomposition and a $W^{1,2}$ weak Dolbeault decomposition, following the $L^2$ weak Kodaira decomposition on Riemannian manifolds. Moreover, if the K\"ahler metric is complete and the  sectional curvature is bounded, the $W^{1,2}$ Bott-Chern decomposition is strictly related to the space of $W^{1,2}$ Bott-Chern harmonic forms, i.e., $W^{1,2}$ smooth differential forms which are in the kernel of an elliptic differential operator of order $4$, called Bott-Chern Laplacian. We also generalize to the non compact case the well known property that on compact K\"ahler manifolds the kernel of the Dolbeault Laplacian and the kernel of the Bott-Chern Laplacian coincide.
\end{abstract}

\section{Introduction}
Let $(M,g)$ be a Riemannian manifold of dimension $n$. Assume, for simplicity, the manifold is oriented, and consider $M$ endowed with the standard Riemannian volume form. Denote by $A^{k}$ the space of smooth $k$-forms, by $A^k_c$ the space of smooth $k$-forms with compact support, and by $L^2A^{k}$ the space of possibly non smooth measurable $k$-forms which are square integrable on $M$. Let $*:A^k\to A^{n-k}$ be the star Hodge operator. Indicate by $d$ the exterior differential on forms, and by $d^*$ its formal adjoint. Define $L^{2}\tilde{\mathcal{H}}^{k}\subset L^2A^{k}$ the subset of $L^2$ forms $\phi$ such that the forms $d\phi$ and $d^*\phi$ are equal to zero in the sense of distributions. 
Kodaira, in 1949, \cite{Ko}, proved the following fundamental orthogonal decomposition of the Hilbert space $L^2A^{k}$:
\begin{theorem}\cite[Theorem 24]{Dr}\label{teo-kod}
Let $(M,g)$ be a orientable Riemannian manifold. Then
\begin{equation*}
L^{2}A^{k}=L^{2}\tilde{\mathcal{H}}^{k}\overset{\perp}{\oplus}\overline{dA^{k-1}_c}\overset{\perp}{\oplus}\overline{d^*A^{k+1}_c}.
\end{equation*}
Moreover, $L^{2}\tilde{\mathcal{H}}^{k}\subset A^k$, and 
\begin{equation*}
L^{2}A^{k}\cap A^k=L^{2}\tilde{\mathcal{H}}^{k}\overset{\perp}{\oplus}\left(\overline{dA^{k-1}_c}\cap A^k\right)\overset{\perp}{\oplus}\left(\overline{d^*A^{k+1}_c}\cap A^k\right).
\end{equation*}
\end{theorem}

After the establishment of Theorem \ref{teo-kod} by Kodaira, this new born $L^2$ Hodge theory was contextualized into the theory of unbounded operators between Hilbert spaces. See, e.g., \cite{C}, \cite{CGM}, \cite{D} for some classical reviews, and, e.g., \cite{A}, \cite{B}, \cite[Chapter VIII]{De} for some modern reviews of the topic. See \cite[Chapter VIII]{De} also for various applications of this theory on complex manifolds, e.g., $L^2$ estimates for the solutions of equations
$\delbar u=v$ on weakly pseudoconvex manifolds.

It is useful to understand under which circumstances the space $L^{2}\tilde{\mathcal{H}}^{k}$ coincides with the space of $L^2$ harmonic forms, since this is true for compact Riemannian manifolds.
Denote by $\Delta_d=dd^*+d^*d$ the Hodge Laplacian.
If the Riemannian manifold $(M,g)$ is complete, given a form $\phi\in L^2A^{k}\cap A^k$, Andreotti and Vesentini, in 1965, \cite[Proposition 7]{AV}, proved that $d\phi=0$ and $d^*\phi=0$ if and only if $\Delta_d\phi=0$, i.e.,
\begin{equation*}
L^{2}\tilde{\mathcal{H}}^{k}=L^{2}{\mathcal{H}}^{k}:=\{\phi\in L^2A^{k}\cap A^k\,|\,\Delta_d\phi=0\}.
\end{equation*}

These results by Kodaira, Andreotti and Vesentini hold without any significant modifications in the contex of Hermitian manifolds $(M,J,g,\omega)$, substituing the operators $d$, $d^*$ and $\Delta_d$ respectively either with $\delbar$, $\delbar^*:=-*\delbar*$ and $\Delta_{\delbar}=\delbar\delbar^*+\delbar^*\delbar$, or with $\del$, $\del^*:=-*\del*$ and $\Delta_{\del}=\del\del^*+\del^*\del$, where $*: A^{p,q}\to  A^{n-p,n-q}$ is the complex anti-linear Hodge operator associated with $g$, and $A^{p,q}$ denotes the space of complex forms of bidegree $(p,q)$.

Kodaira and Spencer, in 1960, \cite{KS}, while developing the theory of deformations of complex structures, introduced the following elliptic and formally self adjoint differential operator of order $4$
\begin{equation*}
\tilde\Delta_{BC} \;:=\;
\del\delbar\delbar^*\del^*+
\delbar^*\del^*\del\delbar+\del^*\delbar\delbar^*\del+\delbar^*\del\del^*\delbar
+\del^*\del+\delbar^*\delbar
\end{equation*}
to prove the stability of the K\"ahler condition under small deformations. Schweitzer, in 2007, \cite{S}, developed a Hodge theory and proved a Hodge decomposition for the operator $\tilde\Delta_{BC}$ on compact complex manifolds, naming it the {\em Bott-Chern Laplacian}, since its kernel turns out to be isomorphic to the Bott-Chern cohomology. The Hodge decomposition proved by Schweitzer is called Bott-Chern decomposition. The kernel of the Bott-Chern Laplacian is called the space of Bott-Chern harmonic forms and it is denoted by $\H^{p,q}_{BC}$. He proved
\begin{theorem}[Bott-Chern decomposition]\cite[Theorem 2.2]{S}\label{teo-schw}
Let $(M,J,g,\omega)$ be a compact Hermitian manifold. Then
\begin{gather*}\label{bc-pieces}
A^{p,q}=\H^{p,q}_{BC}\overset{\perp}{\oplus}{\del\delbar A^{p-1,q-1}}\overset{\perp}{\oplus}{\del^* A^{p+1,q}+\delbar^* A^{p,q+1}}.
\end{gather*}
\end{theorem}

During the last years, Tomassini and the author of the present paper studied $W^{1,2}$ Bott-Chern harmonic forms, namely smooth forms which are in the kernel of the operator $\tilde\Delta_{BC}$ with bounded $W^{1,2}$ norm, on $d$-bounded Stein manifolds \cite{PT1}, and on complete Hermitian manifolds \cite{PT2}. We proved some characterizations of $W^{1,2}$ Bott-Chern harmonic forms and vanishing results following Gromov \cite{G}. In particular, on complete K\"ahler manifolds with bounded sectional curvature, we generalized the classical characterization of Bott-Chern harmonic forms holding on compact K\"ahler manifolds. The result can be viewed as the Bott-Chern analogue of the Theorem by Andreotti and Vesentini we discussed above. Denote by $\lv\cdot\rv$ the standard $L^2$ norm defined on tensors, and by $\nabla$ the Levi-Civita connection.
\begin{theorem}\cite[Theorem 4.4]{PT2}\label{cor-kahler}
Let $(M,J,g,\omega)$ be a complete K\"ahler manifold. Assume that the sectional curvature is bounded. 
Let $\phi\in A^{p,q}$, with $\lv\phi\rv<+\infty$ and $\lv\nabla\phi\rv<+\infty$. Then
\begin{equation*}\label{bc-pieces}
\begin{split}
\tilde{\Delta}_{BC}\phi=0\quad\iff\quad \del\phi=0,\,\delbar\phi=0,\,\del^*\phi=0,\,\delbar^*\phi=0.
\end{split}
\end{equation*}
\end{theorem}
Taking into account Theorem \ref{cor-kahler}, we are motivated to investigate a decomposition of a Sobolev space of differential $(p,q)$-forms, involving the above mentioned space of $W^{1,2}$ Bott-Chern harmonic forms. To do this, we introduce the following $W^{1,2}$ inner product, as in \cite[Section 2]{AV}.

Let $\la\cdot,\cdot\ra$ be the standard $L^2$ inner product defined for $(p,q)$-forms on Hermitian manifolds. For $\alpha,\beta\in A^{p,q}$, set the $W^{1,2}$ inner product
\begin{equation*}
\la\alpha,\beta\ra_2:=\la\alpha,\beta\ra+\la\delbar\alpha,\delbar\beta\ra+\la\delbar^*\alpha,\delbar^*\beta\ra,
\end{equation*}
and denote by $W^{1,2}_2 A^{p,q}$ the completion of the space of $(p,q)$-forms with compact support $A^{p,q}_c$ with respect to the norm $\lv\cdot\rv_2:=\la\cdot,\cdot\ra_2^{\frac12}$.  Define the space
\begin{equation*}
W^{1,2}_2\tilde{\mathcal{H}}^{p,q}_{BC}:=\{\phi\in W^{1,2}_2 A^{p,q}\,|\,\forall\gamma\in A^{*,*}_c\ \la \phi,d^*\gamma\ra_2=\la\phi,\del\delbar\gamma\ra_2=0\}.
\end{equation*}
Then, we are able to prove the following $W^{1,2}$ weak Bott-Chern decomposition. See Theorem \ref{teo-decomp} and Proposition \ref{prop-reg}.
\begin{theorem}\label{thm-main-1}
Let $(M,J,g,\omega)$ be a K\"ahler manifold. Then, we get the following orthogonal decomposition of the Hilbert space $(W^{1,2}_2 A^{p,q},\la\cdot,\cdot\ra_2)$:
\begin{align*}
W^{1,2}_2 A^{p,q}&=W^{1,2}_2\tilde{\mathcal{H}}^{p,q}_{BC}\overset{\perp}{\oplus}\overline{\del\delbar A^{p-1,q-1}_c}\overset{\perp}{\oplus}\overline{\del^* A^{p+1,q}_c+\delbar^* A^{p,q+1}_c}.
\end{align*}
Moreover, $W^{1,2}_2\tilde{\mathcal{H}}^{p,q}_{BC}\subset A^{p,q}$, and 
\begin{gather*}
W^{1,2}_2 A^{p,q}\cap A^{p,q}=W^{1,2}_2\tilde{\mathcal{H}}^{p,q}_{BC}\overset{\perp}{\oplus}\left(\overline{\del\delbar A^{p-1,q-1}_c}\cap A^{p,q}\right)\overset{\perp}{\oplus}\left(\overline{\del^* A^{p+1,q}_c+\delbar^* A^{p,q+1}_c}\cap A^{p,q}\right).
\end{gather*}
\end{theorem}

Define also the space
\begin{equation*}
W^{1,2}_2\tilde{\mathcal{H}}^{p,q}_{\delbar}:=\{\phi\in W^{1,2}_2 A^{p,q}\,|\,\forall\gamma\in A^{*,*}_c\ \la \phi,\delbar^*\gamma\ra_2=\la\phi,\delbar\gamma\ra_2=0\}.
\end{equation*}
Arguing in a similar way as before, we obtain the following $W^{1,2}$ weak Dolbeault decomposition. See Theorem \ref{teo-decomp-dol} and Proposition \ref{prop-reg-dol}.
\begin{theorem}\label{thm-main-2}
Let $(M,J,g,\omega)$ be a K\"ahler manifold. Then we get the following orthogonal decomposition of the Hilbert space $(W^{1,2}_2 A^{p,q},\la\cdot,\cdot\ra_2)$:
\begin{align*}
W^{1,2}_2 A^{p,q}&=W^{1,2}_2\tilde{\mathcal{H}}^{p,q}_{\delbar}\overset{\perp}{\oplus}\overline{\delbar A^{p,q-1}_c}\overset{\perp}{\oplus}\overline{\delbar^* A^{p,q+1}_c}.
\end{align*}
Moreover, $W^{1,2}_2\tilde{\mathcal{H}}^{p,q}_{\delbar}\subset A^{p,q}$, and 
\begin{gather*}
W^{1,2}_2 A^{p,q}\cap A^{p,q}=W^{1,2}_2\tilde{\mathcal{H}}^{p,q}_{\delbar}\overset{\perp}{\oplus}\left(\overline{\delbar A^{p,q-1}_c}\cap A^{p,q}\right)\overset{\perp}{\oplus}\left(\overline{\delbar^* A^{p,q+1}_c}\cap A^{p,q}\right).
\end{gather*}
\end{theorem}

This notes are divided in the following way. In section \ref{hilb}, we briefly recall some concepts from the theory of unbounded operators on a Hilbert space, introduce the maximal and minimal extension of differential operator on a manifold, and state the classical results of elliptic regularity which will be useful in the following.
In section \ref{bg material}, we set the notation of complex and K\"ahler manifolds, and recall the definitions and the main properties of the differential operators which will be studied later.
In section \ref{norms}, we describe four $W^{1,2}$ norms of differential $(p,q)$-forms, which turn to be equivalent on K\"ahler manifolds with bounded sectional curvature. 
In section \ref{decomp}, we prove a rule of integration by parts for the $W^{1,2}$ inner product introduced above, from which we derive our main results, Theorem \ref{thm-main-1} and Theorem \ref{thm-main-2}.
Finally, in section \ref{decomp-compl}, we highlight the relation between these $W^{1,2}$ weak decompositions and the spaces of $W^{1,2}$ Bott-Chern or Dolbeault harmonic forms, on complete K\"ahler manifolds with bounded sectional curvature. We also generalize, to the non compact case, the well known property that on compact K\"ahler manifolds the kernel of the Dolbeault Laplacian and the kernel of the Bott-Chern Laplacian coincide.

We remark that the K\"ahler condition is fundamental for the kind of proof of a $W^{1,2}$ weak Bott-Chern or Dolbeault decomposition presented in this work. It would be interesting to understand if a $W^{1,2}$ weak Bott-Chern or Dolbeault decomposition can be determined in full generality for Hermitian manifolds.

\medskip\medskip
\noindent{\em Acknowledgments.} The author would like to thank Adriano Tomassini, for having introduced this really interesting topic to me, and for many useful discussions in the last years.
The author is also sincerely grateful to Francesco Bei, for his kindness and avaiability, and for having answered a lot of my questions.

\section{Unbounded operators on Hilbert spaces and elliptic regularity}\label{hilb}
We briefly recall some concepts from the theory of unbounded operators on a Hilbert space. If $\H$ is a Hilbert space, the graph of a linear operator $P:\H\to\H$ with domain $\D(P)$ is the set $\{(x,Px)\in\H\times\H\,|\,x\in\D(P)\}$. An operator is \emph{closed} if its graph is a closed subset of $\H\times\H$. By the closed graph theorem, an everywhere defined operator with a closed graph is automatically bounded, therefore when dealing with unbounded operators we need to also keep track of their domain.

An \emph{extension} of $P$ is an operator $P'$ such that $\D(P)\subset\D(P')$ and $Px=P'x$ for every $x\in\D(P)$. An operator is \emph{closable} if the closure of its graph is the graph of a linear operator.

If $\D(P)$ is dense in $\H$, then we say that $P$ is a \emph{densely defined} operator, and we can define the \emph{adjoint} of $P$, indicated by $P^t$. Its domain is
\begin{equation*}
\D(P^t):=\{y\in\H\,|\,x\mapsto\la Px,y\ra\text{ is continuous on }\D(P)\},
\end{equation*}
where here $\la\cdot,\cdot\ra$ denotes the Hermitian inner product of the Hilbert space. If $y\in\D(P^t)$, then $P^ty$ is defined by the relation
\begin{equation*}
\la Px,y\ra=\la x,P^ty\ra\ \forall x\in\D(P).
\end{equation*}
This definition makes $P^t$ a closed operator. If $P$ is closed, then $P^t$ is densely defined and $P^{tt}=P$.

An operator is \emph{symmetric} if $\la Px,y\ra=\la x,Py\ra$ whenever $x,y\in\D(P)$, and \emph{self adjoint} if moreover $\D(P)=\D(P^t)$. A symmetric operator is always closable since its adjoint is a closed extension. An operator is \emph{essentially self adjoint} if it has a unique closed self adjoint extension.

Let $M$ be a differentiable manifold of dimension $m$, and let $E,F$ be $\C$-vector bundles over $M$, with $\rank E=r$, $\rank F=s$.

A $\C$-linear \emph{differential operator} of order $l$ from $E$ to $F$ is a $\C$-linear operator $P:\Gamma(M,E)\to \Gamma(M,F)$ of the form
\begin{equation*}
Pu(x)=\sum_{\lv\alpha\rv\le l}a_\alpha(x)D^\alpha u(x)\ \ \ \forall x\in\Omega,
\end{equation*}
where $E_{|\Omega}\simeq\Omega\times\C^r$, $F_{|\Omega}\simeq\Omega\times\C^s$ are trivialized locally on some open chart $\Omega\subset M$ equipped with local coordinates $x^1,\dots,x^{m}$, and the functions 
\begin{equation*}
a_\alpha(x)=(a_{\alpha ij}(x))_{1\le i\le s,1\le j\le r}
\end{equation*}
are $s\times r$ matrices with smooth coefficients on $\Omega$. Here 
\begin{equation*}
D^\alpha=(\del/\del x^1)^{\alpha_1}\dots(\del/\del x^{m})^{\alpha_{m}},
\end{equation*}
and $u=(u_j)_{1\le j\le r}$, $D^\alpha u=(D^\alpha u_j)_{1\le j\le r}$ are viewed as column matrices. Moreover, we require $a_\alpha\nequiv0$ for some open chart $\Omega\subset M$ and for some $\lv\alpha\rv= l$.

Let $P:\Gamma(M,E)\to \Gamma(M,F)$ be a $\C$-linear differential operator of order $l$ from $E$ to $F$. The \emph{principal symbol} of $P$ is the operator
\begin{equation*}
\sigma_P:T^*M\to \Hom(E,F)\ \ \ (x,\xi)\mapsto\sum_{\lv\alpha\rv= l}a_\alpha(x)\xi^\alpha.
\end{equation*}
We say that $P$ is \emph{elliptic} if $\sigma_P(x,\xi)\in\Hom(E_x,F_x)$ is an isomorphism for every $x\in M$ and $0\ne\xi\in T_x^*M$.

Let $(M,g)$ a Riemannian manifold of dimension $m$. Assume, for simplicity, the manifold is oriented, and consider the standard Riemannian volume form locally given by
\begin{equation*}
\vol(x)=|\det g_{ij}(x)|^{\frac12}dx^1\dots dx^m,
\end{equation*}
where $g(x)=\sum g_{ij}(x)dx^i\otimes dx^j$ for local coordinates $x^1,\dots,x^m$.
Let $E$ be a $\C$-vector bundle over $M$, and take a Hermitian metric $h$ over $E$, i.e., a smooth section of Hermitian inner products on the fibers. The couple $(E,h)$, or simply $E$, will be called a \emph{Hermitian vector bundle}. We define the Hilbert space $L^pE$, $p\ge1$, of global sections $u$ of $E$ with measurable coefficients and finite $L^p$ norm, i.e.,
\begin{equation*}
\lV u\rV_{L^p}:=\left(\int_M| u(x)|^p\vol(x)\right)^{\frac1{p}}<+\infty,
\end{equation*}
where $|\cdot|=(\langle\cdot,\cdot\rangle)^\frac12$ and $\langle\cdot,\cdot\rangle$ is the Hermitian metric on $E$. Note that $L^pE$ can be seen as the completion of $\Gamma_c(M,E)$, the set of smooth section with compact support, with respect to the $L^p$ norm.  We denote by $L^p_{loc}E$ the space of global sections $u$ of $E$ with measurable coefficients such that $fu\in L^pE$ for every smooth function $f\in\cinf_c(M)$ with compact support.
For $p=2$, we denote the corresponding global $L^2$ inner product by 
\begin{equation*}
\la u,v\ra:=\int_M\langle u(x),v(x)\rangle\vol(x).
\end{equation*}
The space $L^2E$ together with $\la\cdot,\cdot\ra$ is an Hilbert space. Denote by $\lv\cdot\rv$ the $L^2$ norm $\lv\cdot\rv_{L^2}$.

Let $E,F$ be Hermitian vector bundles, and let $P:\Gamma(M,E)\to \Gamma(M,F)$ be a differential operator. We define the \emph{formal adjoint}
\begin{equation*}
P^*:\Gamma(M,F)\to \Gamma(M,E)
\end{equation*}
of $P$ by requiring that for all smooth sections $u\in\Gamma(M,E)$ and $v\in\Gamma(M,F)$, then
\begin{equation*}
\la Pu,v\ra=\la u,P^*v\ra
\end{equation*}
whenever $\supp u\cap\supp v$ is compactly contained in $M$.

We remark that the formal adjoint $P^*$ is a differential operator, it always exists and it is unique, see e.g. \cite[Chapter VI, Definition 1.5]{De}. Note that $T^{**}=T$.

Let $E,F$ be Hermitian vector bundles, and let $P:\Gamma(M,E)\to \Gamma(M,F)$ be a differential operator. Then it defines an unbounded linear operator $\tilde{P}:L^2E\to L^2F$ which is densely defined and closable. It is densely defined since its domain cointains the set of smooth sections with compact support $\Gamma_c(M,E)$, and we are going to show two canonical closed extensions of $P$. The \emph{minimal closed extension} $P_{min}$, or \emph{strong extension} $P_{s}$, is defined by taking the closure of the graph of $P$, i.e.,
\begin{equation*}
\D(P_s):=\{u\in L^2E\,|\,\exists\{u_j\}_j\subset \Gamma_c(M,E),\ \exists v\in L^2E,\ u_j\to u,\ Pu_j\to v\},
\end{equation*}
and $P_s(u):=v$. The \emph{maximal closed extension} $P_{max}$, or \emph{weak extension} $P_w$, is defined by letting $P$ act distributionally, i.e.,
\begin{equation*}
\D(P_w):=\{u\in L^2E\,|\,\exists v\in L^2E,\ \forall w\in \Gamma_c(M,E)\ \la v,w\ra=\la u,P^*w\ra\},
\end{equation*}
and $P_w(u):=v$. Note that $\D(P_s)\subset\D(P_w)$. Moreover, it is easy to see $(P^*)^t=P_w$. A densely defined operator and its minimal closed extension have the same adjoint, \cite[Theorem VIII.1]{RS}, therefore $((P^*)_s)^t=P_w$, implying
\begin{equation*}\label{strong-weak}
(P^*)_s=(P_w)^t,\ \ \ (P^*)_w=(P_s)^t.
\end{equation*}
Then, a \emph{formally self adjoint} operator, i.e., $P=P^*$, is essentially self adjoint if and only if $P_s=P_w$, see \cite[Page 256]{RS} for a proof.

Finally, we state the following result about elliptic regularity, for which proof we refer to \cite[Corollary 10.3.10]{N}.
Let $E,F$ be Hermitian vector bundles, and let $P:\Gamma(M,E)\to \Gamma(M,F)$ be a differential operator.
We say that the section $u$ is a \emph{weak solution} of $Pu=v$ if $u,v\in L^1_{loc}(E)$ and 
\begin{equation*}
\la u,P^*w\ra=\la v,w\ra\ \ \ \forall w\in\Gamma_c(M,F).
\end{equation*}
\begin{theorem}\label{ell-reg}
Let $(M,g)$ a orientable Riemannian manifold, and let $E,F$ be Hermitian vector bundles over $M$. Let $P:\Gamma(M,E)\to \Gamma(M,F)$ be an elliptic differential operator. If $u\in L^1_{loc}E$, $u$ is a weak solution of $Pu=v$ and $v$ is smooth, then $u$ must be smooth.
\end{theorem}

\section{Complex and K\"ahler manifolds}\label{bg material}
Let $(M,J,g,\omega)$ be a Hermitian manifold of complex dimension $n$, where $M$ is a smooth manifold of real dimension $2n$, $J$ is a complex structure on $M$, $g$ is a $J$-invariant Riemannian metric on $M$, and $\omega$ denotes the fundamental $(1,1)$-form associated to the metric $g$. We denote by $h$ the Hermitian extension of $g$ on the complexified tangent bundle $T^\C M=TM\otimes_\R\C$, and by the same symbol $g$ the $\C$-bilinear symmetric extension of $g$ on $T^\C M$. Also denote by the same symbol $\omega$ the $\C$-bilinear extension of the fundamental form $\omega$ of $g$ on $T^\C M$. Recall that $h(u,v)=g(u,\bar{v})$ for all $u,v\in T^{1,0}M$, and $\omega(u,v)=g(Ju,v)$ for all $u,v\in TM$.
We denote by $ A^r$ the space of $r$-forms with real values $\Gamma(\Lambda^rM)$, and by $ A^{p,q}$ the space of $(p,q)$-forms with complex values $\Gamma(\Lambda^{p,q}M)$.
We will consider only manifolds without boundary.

Let $(M,J,g,\omega)$ be a Hermitian manifold of dimension $n$ and let $\vol=\frac{\omega^n}{n!}$ be the standard volume form. We consider $M$ endowed with the corresponding Riemannian measure. Given a (possibly non smooth) measurable $(p,q)$-form $\phi$, the pointwise norm $|\phi|$ is defined as $|\phi|=\langle\varphi,\varphi\rangle^\frac12$, where $\langle\cdot,\cdot\rangle$ is the pointwise Hermitian inner product induced by $g$ on the space of $(p,q)$-forms. More generally, we define in the same way $|\cdot|$ and $\langle\cdot,\cdot\rangle$ on tensors. Then, $L^2 A^{p,q}$ is defined as the space of measurable $(p,q)$-forms such that
\begin{equation*}
\lVert \phi\rVert:=\Big(\int_M\vert\phi\vert^2\vol\Big)^\frac12<\infty.
\end{equation*}
The space $L^2 A^{p,q}$, together with the Hermitian product
\begin{equation*}
\la\varphi,\psi\ra :=\int_M\langle\varphi,\psi\rangle \vol,
\end{equation*}
is a Hilbert space. The space $L^2 A^{p,q}$ can be also seen as the completion of $ A^{p,q}_c$, the space of smooth $(p,q)$-forms with compact support, with respect to the norm $\lVert\cdot\rVert$. Again, more generally, we define in the same way $\lv\cdot\rv$ and $\la\cdot,\cdot\ra$ on tensors.

For any given tensor $\varphi$, we also set 
\begin{equation*}
\lVert \varphi\rVert_{L^\infty}:=\sup_{M}\vert\varphi\vert,
\end{equation*}
and we call $\varphi$ {\em bounded} if $
\lVert \varphi\rVert_{L^\infty}<\infty$.

Denoting by $*: A^{p,q}\to  A^{n-p,n-q}$ the complex anti-linear Hodge operator 
associated with $g$, we recall the definitions of the following well known $2$-nd order elliptic and formally self adjoint differential operators
\begin{equation*}
\Delta_{d}:=dd^*+d^*d,\ \ \ \Delta_{\delbar}:=\delbar\delbar^*+\delbar^*\delbar,\ \ \ \Delta_{\del}:=\del\del^*+\del^*\del,
\end{equation*}
which are respectively called {\em Hodge Laplacian}, {\em Dolbeault Laplacian}, and {\em $\del$-Laplacian}, where, as usual
\begin{equation*}
\del^*:=-*\del *,\qquad \delbar^*:=-*\delbar*, \qquad d^*=-*d*,
\end{equation*}
are the formal adjoints respectively of $\del,\delbar,d$.
Moreover, the {\em Bott-Chern Laplacian} and {\em Aeppli Laplacian} 
$\tilde\Delta_{BC}$ and $ \tilde\Delta_{A}$ are the $4$-th order elliptic and formally self adjoint differential operators defined respectively as
\begin{equation*}
\tilde\Delta_{BC} \;:=\;
\del\delbar\delbar^*\del^*+
\delbar^*\del^*\del\delbar+\del^*\delbar\delbar^*\del+\delbar^*\del\del^*\delbar
+\del^*\del+\delbar^*\delbar
\end{equation*}
and
\begin{equation*}
 \tilde\Delta_{A} \;:=\; \del\delbar\delbar^*\del^*+
\delbar^*\del^*\del\delbar+
\del\delbar^*\delbar\del^*+\delbar\del^*\del\delbar^*+
\del\del^*+\delbar\delbar^*.
\end{equation*}
Bott-Chern and Aeppli Laplacians are linked by the duality relation
\begin{equation*}
 *\tilde\Delta_{A}=\tilde\Delta_{BC}*,\ \ \ *\tilde\Delta_{BC}=\tilde\Delta_{A}*.
\end{equation*}
We will be only interested in studying differential $(p,q)$-forms lying in the kernel of the Bott-Chern Laplacian. The same study can be done for the Aeppli Laplacian, using this duality relation when necessary.

If $(M,J,g,\omega)$ is a K\"ahler manifold, i.e., $d\omega=0$, then the Bott-Chern Laplacian and the Aeppli Laplacian can be written in a more concise form. Indeed, by K\"ahler identities, see e.g. \cite[Chapter VI, Theorem 6.4]{De}, we know that $\del$ and $\delbar^*$ anticommute, as well as $\del^*$ and $\delbar$. Moreover, it follows $\Delta_{d}=2\Delta_{\del}=2\Delta_{\delbar}$. Therefore, we derive
\begin{equation*}
\tilde\Delta_{BC}=\Delta_{\delbar}\Delta_{\delbar}+
\del^*\del+\delbar^*\delbar
\end{equation*}
and
\begin{equation*}
 \tilde\Delta_{A}= \Delta_{\delbar}\Delta_{\delbar}+
\del\del^*+\delbar\delbar^*.
\end{equation*}

In the following we will make use of normal holomorphic coordinates on K\"ahler manifolds. We recall that, if $(M,J,g,\omega)$ is a Hermitian manifold, then $g$ is K\"ahler iff
for every $z_0\in M$ there exist local complex coordinates $z^1,\dots,z^n$ centred in $z_0$ such that $g=g_{i\bar{j}}dz^i\otimes d\bar{z}^j+g_{i\bar{j}}d\bar{z}^j\otimes dz^i$ and $g_{i\bar{j}}=\delta_{ij}+[2]$, where $[2]$ indicates terms of order $\ge 2$, which is equivalent to say
\begin{equation*}
\de{g_{i\bar{j}}}{z^k}(z_0)=\de{g_{i\bar{j}}}{\bar{z}^k}(z_0)=0\ \ \forall i,j,k=1,\dots,n.
\end{equation*}

\section{Sobolev spaces on K\"ahler manifolds}\label{norms}

Let $(M,J,g,\omega)$ be a Hermitian manifold of complex dimension $n$. Denote by $\nabla$ the Levi-Civita connection.
On the space of $(p,q)$-forms with compact support $ A^{p,q}_c$ let us consider the following global Hermitian inner products:
\begin{gather*}
\la\alpha,\beta\ra_1:=\la\alpha,\beta\ra+\la\nabla\alpha,\nabla\beta\ra,\\
\la\alpha,\beta\ra_2:=\la\alpha,\beta\ra+\la\delbar\alpha,\delbar\beta\ra+\la\delbar^*\alpha,\delbar^*\beta\ra,\\
\la\alpha,\beta\ra_3:=\la\alpha,\beta\ra+\la\del\alpha,\del\beta\ra+\la\del^*\alpha,\del^*\beta\ra,\\
\la\alpha,\beta\ra_4:=\la\alpha,\beta\ra+\frac12\la\delbar\alpha,\delbar\beta\ra+\frac12\la\delbar^*\alpha,\delbar^*\beta\ra+\frac12\la\del\alpha,\del\beta\ra+\frac12\la\del^*\alpha,\del^*\beta\ra.
\end{gather*}
Denote by $\lv\cdot\rv_i$ the norms defined by $\la\cdot,\cdot\ra_i^\frac12$, for $i=1,2,3,4$.
Define the Sobolev space $W_i^{1,2} A^{p,q}$ as the completion of $ A^{p,q}_c$ with respect to the norms $\lv\cdot\rv_i$, for $i=1,2,3,4$.
By section \ref{hilb}, we may write $W^{1,2}_2 A^{p,q}=\D(\delbar_s)\cap\D((\delbar^*)_s)$, $W^{1,2}_3 A^{p,q}=\D(\del_s)\cap\D((\del^*)_s)$, $W^{1,2}_4 A^{p,q}=W^{1,2}_2 A^{p,q}\cap W^{1,2}_3 A^{p,q}$.

\begin{remark}
Let us consider, on a Hermitian manifold $(M,J,g,\omega)$, possibly non K\"ahler, the Hilbert space $W^{1,2}_2 A^{p,q}$ just defined, as in \cite[Section 2]{AV}. If the Hermitian metric is complete, Andreotti and Vesentini, \cite[Proposition 5]{AV}, proved that $W^{1,2}_2 A^{p,q}$ can be identified with the space of forms $\phi\in L^2 A^{p,q}$ which admit $\delbar\phi\in L^2 A^{p,q+1}$ and $\delbar^*\phi\in L^2 A^{p,q-1}$ in the sense of distributions, i.e., forms $\phi\in L^2 A^{p,q}$ such that there exist $\alpha\in L^2 A^{p,q+1}$ and $\beta\in L^2 A^{p,q-1}$ such that for every $\gamma\in A^{*,*}_c$
\begin{gather*}
\la\alpha,\gamma\ra=\la\phi,\delbar^*\gamma\ra,\\
\la\beta,\gamma\ra=\la\phi,\delbar\gamma\ra.
\end{gather*}
By section \ref{hilb}, this is equivalent to say $W^{1,2}_2 A^{p,q}=\D(\delbar_w)\cap\D((\delbar^*)_w)$. Moreover, they proved that if $\phi\in W^{1,2}_2 A^{p,q}$, then $\delbar_s\phi=\delbar_w\phi$ and $(\delbar^*)_s\phi=(\delbar^*)_w\phi$.
With analogue proofs, we also derive $W^{1,2}_3 A^{p,q}=\D(\del_w)\cap\D((\del^*)_w)$, and if $\phi\in W^{1,2}_3 A^{p,q}$, then $\del_s\phi=\del_w\phi$ and $(\del^*)_s\phi=(\del^*)_w\phi$.
Mutatis mutandis, the same holds also for $W^{1,2}_4 A^{p,q}$.
\end{remark}

Note that on $ A_c^{p,q}$, integrating by parts, we have
\begin{gather}\label{eq-norm-lapl}
\la\alpha,\beta\ra_2=\la\alpha,\beta\ra+\la\alpha,\Delta_{\delbar}\beta\ra,\\
\label{eq-norm-lapl-3}\la\alpha,\beta\ra_3=\la\alpha,\beta\ra+\la\alpha,\Delta_{\del}\beta\ra,\\
\label{eq-norm-lapl-4}\la\alpha,\beta\ra_4=\la\alpha,\beta\ra+\frac12\la\alpha,\Delta_{\del}\beta+\Delta_{\delbar}\beta\ra.
\end{gather}

Now, if we assume that $g$ is K\"ahler, then $\Delta_{d}=2\Delta_{\del}=2\Delta_{\delbar}$ by K\"ahler identities, and
$\la\cdot,\cdot\ra_2=\la\cdot,\cdot\ra_3=\la\cdot,\cdot\ra_4$.
Thus, the norms $\lv\cdot\rv_2$, $\lv\cdot\rv_3$ and $\lv\cdot\rv_4$ are equal on $ A^{p,q}_c$.
Therefore,
\begin{equation*}
W^{1,2}_2A^{p,q}=W^{1,2}_3A^{p,q}=W^{1,2}_4A^{p,q},
\end{equation*}
and the couples $(W^{1,2}_iA^{p,q},\la\cdot,\cdot\ra_i)$ are the same Hilbert spaces for $i=2,3,4$, when $g$ is K\"ahler. 
\begin{remark}
In the following, when proving a result for $\la\cdot,\cdot\ra_j$, with $j$ equal to $2,3$ or $4$ on a K\"ahler manifold, it means the result holds the same way also for $\la\cdot,\cdot\ra_i$, with $i=2,3,4$.
\end{remark}

We now focus out attention on the relation between $\lv\cdot\rv_1$ and $\lv\cdot\rv_i$ for $i=2,3,4$.
Let $(M,J,g,\omega)$ be a K\"ahler manifold of complex dimension $n$.
For any given $\phi\in A^{p,q}$, and for 
\begin{equation*}
A_p=(\alpha_1,\ldots,\alpha_p),\qquad 
B_q=(\beta_1,\ldots,\beta_q)
\end{equation*}
multiindices of length $p$, $q$ respectively, with 
$\alpha_1<\cdots <\alpha_p$ and $\beta_1<\cdots <\beta_q$, write
\begin{equation*}
\phi=\sum_{A_p, B_q}\psi_{A_p\c{B_q}}dz^{A_p}\wedge d\c{z}^{B_q}
\end{equation*}
in local complex coordinates.
Using local normal holomorphic coordinates at $z_0\in M$, we have
\begin{equation*}\label{nabla-normal}
|\nabla\varphi|^2(z_0)=2\sum_{A_p, B_q}\sum_{\gamma=1}^n
\biggl(\biggl|\de{\varphi_{A_p\c{B_q}}}{z^\gamma}\biggr|^2+\biggl|\de{\varphi_{A_p\c{B_q}}}{\c{z}^\gamma}\biggr|^2\biggl)(z_0).
\end{equation*}
Moreover,
\begin{gather*}
\label{del-delbar-normal}
\del\phi(z_0)=\sum_{A_p, B_q}\sum_{\gamma\notin A_p}\de{\varphi_{A_p\c{B_q}}}{z^\gamma}(z_0)dz^{\gamma A_p \c{B_q}}(z_0),\\ 
\delbar\phi(z_0)=\sum_{A_p, B_q}\sum_{\gamma\notin B_q}\de{\varphi_{A_p\c{B_q}}}{\c{z}^\gamma}(z_0)dz^{\c\gamma A_p \c{B_q}}(z_0),\\
\label{del-delbar-star-normal}
\del^*\phi(z_0)=-\sum_{A_p, B_q}\sum_{\gamma\in A_p}\de{\varphi_{\gamma A_p\setminus\{\gamma\}\c{B_q}}}{\c{z}^\gamma}(z_0)dz^{A_p\setminus\{\gamma\} \c{B_q}}(z_0),\\ 
\delbar^*\phi(z_0)=-(-1)^{p}\sum_{A_p, B_q}\sum_{\gamma\in B_q}\de{\varphi_{A_p\c\gamma\c{{B_q}}\setminus\{\c\gamma\}}}{z^\gamma}(z_0)dz^{A_p \c{B_q}\setminus\{\c\gamma\}}(z_0),
\end{gather*}
where the signs of the last two equations can be deduced, e.g., from \cite[Chapter 3, Proposition 2.3]{KM}.
By the previous equations, it follows that $\exists C>0$ depending only on $p,q,n$, such that
\begin{equation}\label{nabla-est}
(|\del\phi|^2+|\delbar\phi|^2+|\del^*\phi|^2+|\delbar^*\phi|^2)(z_0)\le C|\nabla\varphi|^2(z_0).
\end{equation}
Summing up, from equation (\ref{nabla-est}) we get the following result.
\begin{lemma}\label{ineq-kahler}
Let $(M,J,g,\omega)$ be a K\"ahler manifold. Then, $\exists C>0$ such that for all $\phi\in A_c^{p,q}(M)$
\begin{equation}\label{nabla-del}
\lv\phi\rv_4\le C\lv\phi\rv_1.
\end{equation}
\end{lemma}
The converse inequality turns out to hold when the sectional curvature is bounded.
\begin{lemma}\label{equiv-kahler}
Let $(M,J,g,\omega)$ be a K\"ahler manifold. Assume that the sectional curvature is bounded. 
Then $\lV\cdot\rV_1$ and $\lV\cdot\rV_4$ are equivalent.
\end{lemma}
\begin{proof}
The Weitzenb\"ock formula for $\phi\in A^{p,q}_c$ is
\begin{equation}\label{weitz}
\begin{split}
\Delta_{d}\phi=\nabla^*\nabla\phi+R(\phi),
\end{split}
\end{equation}
where $\nabla^*$ is the formal adjoint of $\nabla$, and $R$ denotes an operator of order zero whose coefficients involve the curvature tensor. In particular, since the sectional curvature is bounded, then $R$ is a bounded operator.
We compute the $L^2$ product of  both sides of equation (\ref{weitz}) with the form $\phi$:
\begin{equation*}
\begin{split}
\la\Delta_{d}\phi,\phi\ra=\la\nabla^*\nabla\phi,\phi\ra+\la R(\phi),\phi\ra.
\end{split}
\end{equation*}
Integrating by parts, we get
\begin{equation*}
\begin{split}
\lV d\phi\rV^2+\lV d^*\phi\rV^2=\lV\nabla\phi\rV^2+\la R(\phi),\phi\ra.
\end{split}
\end{equation*}
Since $R$ is a bounded operator, we derive there exists $C>0$ not depending on $\phi$ such that
\begin{equation*}
\begin{split}
|\la R(\phi),\phi\ra|\le C \lV\phi\rV^2,
\end{split}
\end{equation*}
and
\begin{equation*}
\begin{split}
\lV d\phi\rV^2+\lV d^*\phi\rV^2\ge\lV\nabla\phi\rV^2-C\lV\phi\rV^2,
\end{split}
\end{equation*}
which implies
\begin{equation*}
\begin{split}
\lV \phi\rV_1^2\le2(C+1)\lV\phi\rV^2_4.
\end{split}
\end{equation*}
This, together with (\ref{nabla-del}), ends the proof.
\end{proof}

Lemma \ref{equiv-kahler} implies that 
\begin{equation*}
W^{1,2}_1A^{p,q}=W^{1,2}_2A^{p,q}=W^{1,2}_3A^{p,q}=W^{1,2}_4A^{p,q},
\end{equation*}
for a K\"ahler manifold with bounded sectional curvature.

\section{$W^{1,2}$ weak Bott-Chern and Dolbeault decompositions}\label{decomp}
In this section we prove our main results, i.e., $W^{1,2}$ weak Bott-Chern and Dolbeault decompositions.
The following lemma is essential to derive these decompositions.
\begin{lemma}[$W^{1,2}$ integration by parts]\label{adj-kahler}
Let $(M,J,g,\omega)$ be a K\"ahler manifold. 
Let $\alpha,\beta\in A^{*,*}$. If at least one between $\alpha$ and $\beta$ has compact support, then
\begin{equation*}\label{adj-formula}
\begin{split}
\la\del\alpha,\beta\ra_2=\la\alpha,\del^*\beta\ra_2,\ \ \ \la\delbar\alpha,\beta\ra_2=\la\alpha,\delbar^*\beta\ra_2.
\end{split}
\end{equation*}
\end{lemma}
\begin{proof}
By K\"ahler identities, we know that $\del$ and $\delbar^*$ anticommute, as well as $\del^*$ and $\delbar$. Using these facts and integrating by parts, we conclude as follows:
\begin{equation*}
\begin{split}
\la\del\alpha,\beta\ra_2&=\la\del\alpha,\beta\ra+\la\delbar\del\alpha,\delbar\beta\ra+\la\delbar^*\del\alpha,\delbar^*\beta\ra\\
&=\la\del\alpha,\beta\ra-\la\del\delbar\alpha,\delbar\beta\ra-\la\del\delbar^*\alpha,\delbar^*\beta\ra\\
&=\la\alpha,\del^*\beta\ra-\la\delbar\alpha,\del^*\delbar\beta\ra-\la\delbar^*\alpha,\del^*\delbar^*\beta\ra\\
&=\la\alpha,\del^*\beta\ra+\la\delbar\alpha,\delbar\del^*\beta\ra+\la\delbar^*\alpha,\delbar^*\del^*\beta\ra\\
&=\la\alpha,\del^*\beta\ra_2.
\end{split}
\end{equation*}
To prove the second equality, recall that $\la\cdot,\cdot\ra_2=\la\cdot,\cdot\ra_3$ and proceed as before:
\begin{equation*}
\begin{split}
\la\delbar\alpha,\beta\ra_3&=\la\delbar\alpha,\beta\ra+\la\del\delbar\alpha,\del\beta\ra+\la\del^*\delbar\alpha,\del^*\beta\ra\\
&=\la\delbar\alpha,\beta\ra-\la\delbar\del\alpha,\del\beta\ra-\la\delbar\del^*\alpha,\del^*\beta\ra\\
&=\la\alpha,\delbar^*\beta\ra-\la\del\alpha,\delbar^*\del\beta\ra-\la\del^*\alpha,\delbar^*\del^*\beta\ra\\
&=\la\alpha,\delbar^*\beta\ra+\la\del\alpha,\del\delbar^*\beta\ra+\la\del^*\alpha,\del^*\delbar^*\beta\ra\\
&=\la\alpha,\delbar^*\beta\ra_3.\qedhere
\end{split}
\end{equation*}
\end{proof}
Thanks to Lemma \ref{adj-kahler}, the formal adjoint operators of $\del$, $\delbar$, $d$, with respect to the $W^{1,2}_i A^{p,q}$-norms, for $i=2,3,4$, are the same usual formal adjoint operators $\del^*$, $\delbar^*$, $d^*$  computed with respect to the $L^{2} A^{p,q}$-norm, on K\"ahler manifolds.

Let $(M,J,g,\omega)$ be a Hermitian manifold of complex dimension $n$. 
Given $P:A^{p,q}\to A^{r,s}$ a differential operator, then it defines an unbounded linear operator $\tilde{P_i}:W^{1,2}_i A^{p,q}\to W^{1,2}_i A^{r,s}$, for $i=1,2,3,4$, which is densely defined and closable, as in the $L^2$ case. We define
\begin{equation*}
\D(P_{w,i}):=\{u\in W^{1,2}_i A^{p,q}\,|\,\exists v\in W^{1,2}_i A^{r,s},\ \forall w\in  A^{r,s}_c\ \la v,w\ra_i=\la u,P^*w\ra_i\},
\end{equation*}
and set $P_{w,i}(u):=v$. The operator $P_{w,i}:W^{1,2}_i A^{p,q}\to W^{1,2}_i A^{r,s}$ is a closed and densely defined operator, which extends $P$.

In the following, if $P:W^{1,2}_i A^{*,*}\to W^{1,2}_i A^{*,*}$ is an operator, by $\Ker P$ we denote the space $\Ker P\cap W^{1,2}_i A^{p,q}$, when the bi-gradation $(p,q)$ is clear.

\subsection{Bott-Chern decomposition of the space $W^{1,2}_2 A^{p,q}$}

Define the space
\begin{equation*}
W^{1,2}_i\tilde{\mathcal{H}}^{p,q}_{BC}:=\ker d_{w,i}\cap\ker (\delbar^*\del^*)_{w,i},
\end{equation*}
i.e., 
\begin{equation*}
W^{1,2}_i\tilde{\mathcal{H}}^{p,q}_{BC}=\{\phi\in W^{1,2}_i A^{p,q}\,|\,\forall\gamma\in A^{*,*}_c\ \la \phi,d^*\gamma\ra_i=\la\phi,\del\delbar\gamma\ra_i=0\}.
\end{equation*}
We can now prove the analogue of Theorem \ref{teo-kod} by Kodaira in the $W^{1,2}$ Bott-Chern case. 

\begin{theorem}[$W^{1,2}$ weak Bott-Chern decomposition]\label{teo-decomp}
Let $(M,J,g,\omega)$ be a K\"ahler manifold. Then we get the following orthogonal decomposition of the Hilbert space $(W^{1,2}_2 A^{p,q},\la\cdot,\cdot\ra_2)$:
\begin{align}
W^{1,2}_2 A^{p,q}&=W^{1,2}_2\tilde{\mathcal{H}}^{p,q}_{BC}\overset{\perp}{\oplus}\overline{\del\delbar A^{p-1,q-1}_c}\overset{\perp}{\oplus}\overline{\del^* A^{p+1,q}_c+\delbar^* A^{p,q+1}_c},\label{bc-decomp}\\
\Ker d_{w,2}&=W^{1,2}_2\tilde{\mathcal{H}}^{p,q}_{BC}\overset{\perp}{\oplus}\overline{\del\delbar A^{p-1,q-1}_c}.\label{ker-bc-decomp}
\end{align}
\end{theorem}
\begin{proof} 
First of all, note that the spaces $\overline{\del\delbar A^{p-1,q-1}_c}$ and $\overline{\del^* A^{p+1,q}_c+\delbar^* A^{p,q+1}_c}$ are orthogonal. Indeed, by Lemma \ref{adj-kahler} it is immediate to show that ${\del\delbar A^{p-1,q-1}_c}$ and ${\del^* A^{p+1,q}_c+\delbar^* A^{p,q+1}_c}$ are orthogonal.
Set
\begin{equation*}
X=\overline{\del\delbar A^{p-1,q-1}_c}\overset{\perp}{\oplus}\overline{\del^* A^{p+1,q}_c+\delbar^* A^{p,q+1}_c},
\end{equation*}
which is a closed subspace. Therefore we get the orthogonal decomposition of the Hilbert space $W^{1,2}_2 A^{p,q}=X\overset{\perp}{\oplus}X^{\perp}$.
Note that
\begin{equation*}
X^{\perp}=\Big({\del\delbar A^{p-1,q-1}_c}\Big)^{\perp}\cap\Big({\del^* A^{p+1,q}_c+\delbar^* A^{p,q+1}_c}\Big)^{\perp}=W^{1,2}_2\tilde{\mathcal{H}^{p,q}_{BC}}
\end{equation*}
by definition. This proves equation (\ref{bc-decomp}). Equation (\ref{ker-bc-decomp}) follows by intersecting equation (\ref{bc-decomp}) with $\Ker d_{w,2}$.
\end{proof}

It remains to understand the regularity of the spaces involved in the decomposition of Theorem \ref{teo-decomp}.
\begin{theorem}[Bott-Chern Regularity]\label{prop-reg}
Let $(M,J,g,\omega)$ be a Hermitian manifold. Then, for $i=2,3,4$, we get the characterization
\begin{gather*}
W^{1,2}_i\tilde{\mathcal{H}}^{p,q}_{BC}=\{\phi\in A^{p,q}\,|\,d\phi=0,\ \delbar^*\del^*\phi=0,\ \lv\phi\rv_i<\infty\}.
\end{gather*}
Moreover, if $g$ K\"ahler, we also get the decomposition 
\begin{gather*}
W^{1,2}_2 A^{p,q}\cap A^{p,q}=W^{1,2}_2\tilde{\mathcal{H}}^{p,q}_{BC}\overset{\perp}{\oplus}\left(\overline{\del\delbar A^{p-1,q-1}_c}\cap A^{p,q}\right)\overset{\perp}{\oplus}\left(\overline{\del^* A^{p+1,q}_c+\delbar^* A^{p,q+1}_c}\cap A^{p,q}\right).
\end{gather*}
\end{theorem}
\begin{proof} 
For $i=2$, let $\alpha_1\in W^{1,2}_2\tilde{\mathcal{H}}^{p,q}_{BC}$.
For every $\gamma\in A^{p,q}_c$, keeping in mind the very definitions of $W^{1,2}_2\tilde{\mathcal{H}}^{p,q}_{BC}$ and $\tilde{\Delta}_{BC}$, we get
\begin{equation*}\label{eq-ort}
\begin{split}
\la\alpha_1,\tilde{\Delta}_{BC}\gamma\ra_2=0.
\end{split}
\end{equation*}
If we see $\alpha_1$ as a $W^{1,2}_2$-limit of a sequence of compactly supported smooth forms $(\alpha_1)_\nu\in A^{p,q}_c$ and apply equation (\ref{eq-norm-lapl}), we get
\begin{equation*}
\la\alpha_1,(1+\Delta_{\delbar})\tilde{\Delta}_{BC}\gamma\ra=0.
\end{equation*}
That is, $\alpha_1$ is a weak solution of $\tilde{\Delta}_{BC}(1+\Delta_{\delbar})\alpha_1=0$.
Since $\tilde{\Delta}_{BC}(1+\Delta_{\delbar})$ is elliptic, then $\alpha_1$ is smooth by Theorem \ref{ell-reg}.
Furthermore, by the very definition of $W^{1,2}_2\tilde{\mathcal{H}}^{p,q}_{BC}$, we immediately derive
\begin{equation*}
d\alpha_1=0,\ \ \ \delbar^*\del^*\alpha_1=0,\ \ \ \tilde{\Delta}_{BC}\alpha_1=0.
\end{equation*}
For $i=3,4$, it suffices using respectively equations (\ref{eq-norm-lapl-3}),(\ref{eq-norm-lapl-4}) instead of equation (\ref{eq-norm-lapl}).

To finish the proof, let $\alpha=\alpha_1+\alpha_2+\alpha_3$, with $\alpha\in W^{1,2}_2 A^{p,q}\cap A^{p,q}$, $\alpha_1\in W^{1,2}_2\tilde{\mathcal{H}}^{p,q}_{BC}$, $\alpha_2\in \overline{\del\delbar A^{p-1,q-1}_c}$ and $\alpha_3\in \overline{\del^* A^{p+1,q}_c+\delbar^* A^{p,q+1}_c}$. We have to prove that $\alpha_2,\alpha_3$ are smooth.
Note that $d_w\alpha_2=0$ and $(\delbar^*\del^*)_w\alpha_3=0$, i.e., for all $\gamma\in A^{*,*}_c$
\begin{gather*}
\la\alpha_2,d^*\gamma\ra=0,\\
\la\alpha_3,\del\delbar\gamma\ra=0.
\end{gather*}
Also note that $\alpha_2$ is a weak solution of $\tilde{\Delta}_{BC}\alpha_2=\del\delbar\delbar^*\del^*\alpha$, indeed
\begin{gather*}
\la\alpha_2,\tilde{\Delta}_{BC}\gamma\ra=\la\alpha_2,\del\delbar\delbar^*\del^*\gamma\ra=\la\alpha,\del\delbar\delbar^*\del^*\gamma\ra=\la\del\delbar\delbar^*\del^*\alpha,\gamma\ra,
\end{gather*}
and  $\alpha_3$ is a weak solution of $\tilde{\Delta}_{BC}\alpha_3=\delbar^*\del^*\del\delbar\alpha+\del^*\delbar\delbar^*\del\alpha+\delbar^*\del\del^*\delbar\alpha
+\del^*\del\alpha+\delbar^*\delbar\alpha$, indeed
\begin{align*}
\la\alpha_3,\tilde{\Delta}_{BC}\gamma\ra&=\la\alpha_3,\delbar^*\del^*\del\delbar\gamma+\del^*\delbar\delbar^*\del\gamma+\delbar^*\del\del^*\delbar\gamma
+\del^*\del\gamma+\delbar^*\delbar\gamma\ra\\
&=\la\alpha,\delbar^*\del^*\del\delbar\gamma+\del^*\delbar\delbar^*\del\gamma+\delbar^*\del\del^*\delbar\gamma
+\del^*\del\gamma+\delbar^*\delbar\gamma\ra\\
&=\la\delbar^*\del^*\del\delbar\alpha+\del^*\delbar\delbar^*\del\alpha+\delbar^*\del\del^*\delbar\alpha
+\del^*\del\alpha+\delbar^*\delbar\alpha,\gamma\ra.
\end{align*}
By elliptic regularity, i.e., Theorem \ref{ell-reg}, it follows that since $\alpha$ is smooth, then $\alpha_2,\alpha_3$ are smooth.
\end{proof}

\begin{remark}\label{rem-l2}
In this work we are interested in a $W^{1,2}$ weak Bott-Chern decomposition. Nonetheless, we remark that with the same structure of proof of Theorem \ref{teo-decomp}, substituing the application of Lemma \ref{adj-kahler} with the classical $L^2$ integration by parts, it is possible to show the following $L^2$ weak Bott-Chern decomposition.

Let $(M,J,g,\omega)$ be a Hermitian manifold. Then we get the following orthogonal decomposition of the Hilbert space $(L^{2} A^{p,q},\la\cdot,\cdot\ra)$:
\begin{align*}
L^{2} A^{p,q}&=L^{2}\tilde{\mathcal{H}}^{p,q}_{BC}\overset{\perp}{\oplus}\overline{\del\delbar A^{p-1,q-1}_c}\overset{\perp}{\oplus}\overline{\del^* A^{p+1,q}_c+\delbar^* A^{p,q+1}_c},\\
\Ker d_{w}&=L^{2}\tilde{\mathcal{H}}^{p,q}_{BC}\overset{\perp}{\oplus}\overline{\del\delbar A^{p-1,q-1}_c},
\end{align*}
where $L^{2}\tilde{\mathcal{H}}^{p,q}_{BC}:=\ker d_{w}\cap\ker (\delbar^*\del^*)_{w}$, and here the closure and the orthogonal symbols are intended with respect to the $L^2$ inner product. The regularity result holds the same way in the $L^2$ case, and in particular $L^{2}\tilde{\mathcal{H}}^{p,q}_{BC}\subset A^{p,q}$. 

However, in the $L^2$ case, it is not clear how the space $L^{2}\tilde{\mathcal{H}}^{p,q}_{BC}$ is related to the space of $L^2$ Bott-Chern harmonic forms $L^2\H^{p,q}_{BC}$, namely the space of smooth $L^2$ forms $\phi$ satisfying $\tilde{\Delta}_{BC}\phi=0$. It would be interesting to find geometric assumptions on a Hermitian manifold yielding a link between the spaces $L^{2}\tilde{\mathcal{H}}^{p,q}_{BC}$ and $L^2\H^{p,q}_{BC}$ .
\end{remark}

\subsection{Dolbeault decomposition of the space $W^{1,2}_2 A^{p,q}$}

Set
\begin{equation*}
W^{1,2}_i\tilde{\mathcal{H}}^{p,q}_{\delbar}:=\ker \delbar_{w,i}\cap\ker (\delbar^*)_{w,i},
\end{equation*}
i.e., 
\begin{equation*}
W^{1,2}_i\tilde{\mathcal{H}}^{p,q}_{\delbar}=\{\phi\in W^{1,2}_i A^{p,q}\,|\,\forall\gamma\in A^{*,*}_c\ \la \phi,\delbar^*\gamma\ra_i=\la\phi,\delbar\gamma\ra_i=0\}.
\end{equation*}
Arguing in the same way as before, we obtain the following $W^{1,2}$ weak Dolbeault decomposition.

\begin{theorem}[$W^{1,2}$ weak Dolbeault decomposition]\label{teo-decomp-dol}
Let $(M,J,g,\omega)$ be a K\"ahler manifold. Then we get the following orthogonal decomposition of the Hilbert space $(W^{1,2}_2 A^{p,q},\la\cdot,\cdot\ra_2)$:
\begin{align*}
W^{1,2}_2 A^{p,q}&=W^{1,2}_2\tilde{\mathcal{H}}^{p,q}_{\delbar}\overset{\perp}{\oplus}\overline{\delbar A^{p,q-1}_c}\overset{\perp}{\oplus}\overline{\delbar^* A^{p,q+1}_c},\\
\Ker \delbar_{w,2}&=W^{1,2}_2\tilde{\mathcal{H}}^{p,q}_{BC}\overset{\perp}{\oplus}\overline{\delbar A^{p,q-1}_c}.
\end{align*}
\end{theorem}
\begin{proof} 
First of all, note that the spaces $\overline{\delbar A^{p,q-1}_c}$ and $\overline{\delbar^* A^{p,q+1}_c}$ are orthogonal, by Lemma \ref{adj-kahler}.
Set
\begin{equation*}
X=\overline{\delbar A^{p,q-1}_c}\overset{\perp}{\oplus}\overline{\delbar^* A^{p,q+1}_c},
\end{equation*}
which is a closed subspace. Therefore we get the orthogonal decomposition of the Hilbert space $W^{1,2}_2 A^{p,q}=X\overset{\perp}{\oplus}X^{\perp}$.
Note that
\begin{equation*}
X^{\perp}=\Big({\delbar A^{p,q-1}_c}\Big)^{\perp}\cap\Big({\delbar^* A^{p,q+1}_c}\Big)^{\perp}=W^{1,2}_2\tilde{\mathcal{H}^{p,q}_{\delbar}}
\end{equation*}
by definition.
\end{proof}

Concerning the regularity of the spaces involved in the decomposition of Theorem \ref{teo-decomp-dol}, we get
\begin{theorem}[Dolbeault Regularity]\label{prop-reg-dol}
Let $(M,J,g,\omega)$ be a Hermitian manifold. Then, for $i=2,3,4$, we get the characterization
\begin{gather*}
W^{1,2}_i\tilde{\mathcal{H}}^{p,q}_{\delbar}=\{\phi\in A^{p,q}\,|\,\delbar\phi=0,\ \delbar^*\phi=0,\ \lv\phi\rv_i<\infty\}.
\end{gather*}
Moreover, if $g$ is K\"ahler, we also get the decomposition 
\begin{gather*}
W^{1,2}_2 A^{p,q}\cap A^{p,q}=W^{1,2}_2\tilde{\mathcal{H}}^{p,q}_{\delbar}\overset{\perp}{\oplus}\left(\overline{\delbar A^{p,q-1}_c}\cap A^{p,q}\right)\overset{\perp}{\oplus}\left(\overline{\delbar^* A^{p,q+1}_c}\cap A^{p,q}\right).
\end{gather*}
\end{theorem}
\begin{proof} 
For $i=2$, let $\alpha_1\in W^{1,2}_2\tilde{\mathcal{H}}^{p,q}_{\delbar}$.
For every $\gamma\in A^{p,q}_c$, keeping in mind the very definitions of $W^{1,2}_2\tilde{\mathcal{H}}^{p,q}_{\delbar}$ and ${\Delta}_{\delbar}$, we get
\begin{equation*}\label{eq-ort-dol}
\begin{split}
\la\alpha_1,{\Delta}_{\delbar}\gamma\ra_2=0.
\end{split}
\end{equation*}
If we see $\alpha_1$ as a $W^{1,2}_2$-limit of a sequence of compactly supported smooth forms $(\alpha_1)_\nu\in A^{p,q}_c$ and apply equation (\ref{eq-norm-lapl}), we get
\begin{equation*}
\la\alpha_1,(1+\Delta_{\delbar})\Delta_{\delbar}\gamma\ra=0.
\end{equation*}
That is, $\alpha_1$ is a weak solution of $\Delta_{\delbar}(1+\Delta_{\delbar})\alpha_1=0$.
Since $\Delta_{\delbar}(1+\Delta_{\delbar})$ is elliptic, then $\alpha_1$ is smooth by Theorem \ref{ell-reg}.
Furthermore, by the very definition of $W^{1,2}_2\tilde{\mathcal{H}}^{p,q}_{BC}$, we immediately derive
\begin{equation*}
\delbar\alpha_1=0,\ \ \ \delbar^*\alpha_1=0,\ \ \ {\Delta}_{\delbar}\alpha_1=0.
\end{equation*}
For $i=3,4$ it suffices using respectively equations (\ref{eq-norm-lapl-3}),(\ref{eq-norm-lapl-4}) instead of equation (\ref{eq-norm-lapl}).

To finish the proof, let $\alpha=\alpha_1+\alpha_2+\alpha_3$, with $\alpha\in W^{1,2}_2 A^{p,q}\cap A^{p,q}$, $\alpha_1\in W^{1,2}_2\tilde{\mathcal{H}}^{p,q}_{BC}$, $\alpha_2\in \overline{\delbar A^{p,q-1}_c}$ and $\alpha_3\in \overline{\delbar^* A^{p,q+1}_c}$. We have to prove that $\alpha_2,\alpha_3$ are smooth.
Note that $\delbar_w\alpha_2=0$ and $(\delbar^*)_w\alpha_3=0$, i.e., for all $\gamma\in A^{*,*}_c$
\begin{gather*}
\la\alpha_2,\delbar^*\gamma\ra=0,\\
\la\alpha_3,\delbar\gamma\ra=0.
\end{gather*}
Also note that $\alpha_2$ is a weak solution of $\Delta_{\delbar}\alpha_2=\delbar\delbar^*\alpha$, indeed
\begin{gather*}
\la\alpha_2,\Delta_{\delbar}\gamma\ra=\la\alpha_2,\delbar\delbar^*\gamma\ra=\la\alpha,\delbar\delbar^*\gamma\ra=\la\delbar\delbar^*\alpha,\gamma\ra,
\end{gather*}
and  $\alpha_3$ is a weak solution of $\Delta_{\delbar}\alpha_3=\delbar^*\delbar\alpha$, indeed
\begin{align*}
\la\alpha_3,\Delta_{\delbar}\gamma\ra&=\la\alpha_3,\delbar^*\delbar\gamma\ra=\la\alpha,\delbar^*\delbar\gamma\ra=\la\delbar^*\delbar\alpha,\gamma\ra.
\end{align*}
By elliptic regularity, i.e., Theorem \ref{ell-reg}, it follows that since $\alpha$ is smooth, then $\alpha_2,\alpha_3$ are smooth.
\end{proof}

\section{Complete K\"ahler manifolds with bounded sectional curvature}\label{decomp-compl}

In this section, we gather the known relations between the spaces of $W^{1,2}$ Bott-Chern harmonic forms and the spaces of $W^{1,2}$, or $L^2$, Dolbeault harmonic forms, and the relations between these spaces of forms and the $W^{1,2}$ weak decompositions just proved.

Let $(M,J,g,\omega)$ be a Hermitian manifold. For $i=1,2,3,4$, the spaces of $W^{1,2}_i$ Bott-Chern and Dolbeault harmonic forms are defined respectively as
\begin{align*}
W^{1,2}_i\mathcal{H}^{p,q}_{BC}:=\left\{\varphi\in A^{p,q}\,\,\,\vert\,\,\,\tilde\Delta_{BC}\varphi=0,\ \lVert \varphi\rVert_{i}<\infty\right\},\\
W^{1,2}_i\mathcal{H}^{p,q}_{\delbar}:=\left\{\varphi\in A^{p,q}\,\,\,\vert\,\,\,\Delta_{\delbar}\varphi=0,\ \lVert \varphi\rVert_{i}<\infty\right\}.
\end{align*}
If the metric is complete, then by \cite[Proposition 7]{AV} and the characterization of Theorem \ref{prop-reg-dol}, we get, for $i=2,3,4$,
\begin{gather*}
W^{1,2}_i\tilde{\mathcal{H}}^{p,q}_{\delbar}=W^{1,2}_i{\mathcal{H}}^{p,q}_{\delbar}.
\end{gather*}

Let $(M,J,g,\omega)$ be a complete K\"ahler manifold with bounded sectional curvature. Theorem \ref{prop-reg} tells us
\begin{gather*}
W^{1,2}_2\tilde{\mathcal{H}}^{p,q}_{BC}\subset W^{1,2}_2\mathcal{H}^{p,q}_{BC}.
\end{gather*}
Then, Theorem \ref{cor-kahler}, together with Lemma \ref{equiv-kahler} shows
\begin{gather*}
W^{1,2}_2\tilde{\mathcal{H}}^{p,q}_{BC}\supset W^{1,2}_2\mathcal{H}^{p,q}_{BC},
\end{gather*}
thus yielding 
\begin{gather*}
W^{1,2}_2\tilde{\mathcal{H}}^{p,q}_{BC}= W^{1,2}_2\mathcal{H}^{p,q}_{BC}.
\end{gather*}

Furthermore, Theorem \ref{cor-kahler} and Lemma \ref{equiv-kahler} also imply, for $i=1,2,3,4$,
\begin{gather*}
W^{1,2}_i{\mathcal{H}}^{p,q}_{BC}= W^{1,2}_i\mathcal{H}^{p,q}_{\delbar},
\end{gather*}
generalizing, to the non compact case, the well known property that on compact K\"ahler manifolds the kernel of the Dolbeault Laplacian and the kernel of the Bott-Chern Laplacian coincide.
Investigating a little more in this direction, let us set 
\begin{gather*}
L^2\mathcal{H}^{p,q}_{\delbar}:=\left\{\varphi\in A^{p,q}\,\,\,\vert\,\,\,\Delta_{\delbar}\varphi=0,\ \lVert \varphi\rVert<\infty\right\}.
\end{gather*}
Theorem \ref{cor-kahler} implies $W^{1,2}_1{\mathcal{H}}^{p,q}_{BC}\subset L^2\mathcal{H}^{p,q}_{\delbar}$ on complete  K\"ahler manifolds with bounded sectional curvature. The following Corollary shows that this inclusion is, in fact, an equality.
Arguing like in \cite[Lemma 3.10]{HuTa}, one gets:
\begin{lemma}\label{lemma-weitz}
Let $(M,J,g,\omega)$ be a complete K\"ahler manifold. Assume that the sectional curvature is bounded. If $\phi\in L^2\mathcal{H}^{p,q}_{\delbar}$, then $\lVert \varphi\rVert_1<+\infty$, i.e., $L^2\mathcal{H}^{p,q}_{\delbar}=W^{1,2}_1\mathcal{H}^{p,q}_{\delbar}$.
\end{lemma}
\begin{proof}
The Weitzenb\"ock formula for $\phi\in A^{p,q}$ is
\begin{equation*}
\begin{split}
\Delta_{d}\phi=\nabla^*\nabla\phi+R(\phi),
\end{split}
\end{equation*}
where $R$ denotes an operator of order zero whose coefficients involve the curvature tensor. In particular, since the sectional curvature is bounded, then $R$ is a bounded operator.
By K\"ahler identities, we know $\Delta_{d}=2\Delta_{\delbar}$. Therefore, if $\phi\in L^2\mathcal{H}^{p,q}_{\delbar}$, then
\begin{equation*}
\begin{split}
0=\nabla^*\nabla\phi+R(\phi).
\end{split}
\end{equation*}
Since the metric is complete, there exists a sequence of compact subsets $\{K_\nu\}_{\nu\in\N}$, such that $\cup_\nu K_\nu=M$, $K_\nu\subset K_{\nu+1}$ and a sequence of smooth cut-off functions $\{f_\nu\}_{\nu\in\N}$ such that $0\le f_\nu\le1$, $f_\nu=1$ in a neighborhood of $K_\nu$, $\supp f_\nu\subset K_{\nu+1}$, and $|\nabla f_\nu|\le1$. For this last fact we refer, e.g., to \cite[Chap. VIII, Lemma 2.4]{De}.
For every $\nu\in\N$, we compute
\begin{equation*}
\begin{split}
0&=\la\nabla^*\nabla\phi+R(\phi),f_\nu^2\phi\ra\\
&=\la\nabla\phi,\nabla(f_\nu^2\phi)\ra+\la R(\phi),f_\nu^2\phi\ra\\
&=\la\nabla\phi,2f_\nu\nabla f_\nu\otimes\phi\ra+\la\nabla\phi,f_\nu^2\nabla\phi\ra+\la R(\phi),f_\nu^2\phi\ra\\
&=2\la f_\nu\nabla\phi,\nabla f_\nu\otimes\phi\ra+\lv f_\nu\nabla\phi\rv^2+\la R(\phi),f_\nu^2\phi\ra.
\end{split}
\end{equation*}
Moreover,
\begin{equation*}
\begin{split}
\lv \nabla(f_\nu \phi)\rv^2&=\la\nabla f_\nu\otimes\phi+f_\nu\nabla\phi,\nabla f_\nu\otimes\phi+f_\nu\nabla\phi\ra\\
&=\lv\nabla f_\nu\otimes\phi\rv^2+\lv f_\nu\nabla\phi\rv^2+2\la f_\nu\nabla\phi,\nabla f_\nu\otimes\phi\ra\\
&=\lv\nabla f_\nu\otimes\phi\rv^2-\la R(\phi),f_\nu^2\phi\ra\\
&\le C\lv\phi\rv^2,
\end{split}
\end{equation*}
for some $C>0$. Therefore, by the Fatou's lemma, it follows
\begin{equation*}
\lv\nabla\phi\rv^2\le\liminf_{\nu\to\infty}\lv\nabla(f_\nu \phi)\rv^2\le C\lv\phi\rv^2,
\end{equation*}
which implies $\lv\phi\rv_1<+\infty$.
\end{proof}
We immediately derive
\begin{corollary}\label{prop-kernel}
Let $(M,J,g,\omega)$ be a complete K\"ahler manifold of complex dimension $n$. Assume that the sectional curvature is bounded. Then $W^{1,2}_i\mathcal{H}^{p,q}_{BC}=L^2\mathcal{H}^{p,q}_{\delbar}$, for $i=1,2,3,4$.
\end{corollary}

\begin{remark}\label{rem-del}
Note that Theorems \ref{teo-decomp-dol}, \ref{prop-reg-dol}, Lemma \ref{lemma-weitz} and Corollary \ref{prop-kernel} still hold if we substitute the operators $\delbar,\delbar^*,\Delta_{\delbar}$ respectivly with $\del,\del^*,\Delta_{\del}$.
\end{remark}

\end{document}